\documentclass[11pt]{amsart}
\usepackage{amscd,amssymb,graphicx}
\usepackage{pinlabel}

\begin{document}

\newtheorem{thm}{Theorem}
\newtheorem*{thm*}{Theorem}
\newtheorem{cor}[thm]{Corollary}
\newtheorem{lemma}[thm]{Lemma}

\theoremstyle{definition}
\newtheorem{quest}[thm]{Question}
\newtheorem{ex}{Example}

\newcommand{\co}{\colon\,}
\newcommand{\bH}{\mathbb H}
\newcommand{\bN}{\mathbb N}
\newcommand{\bR}{\mathbb R}
\newcommand{\bZ}{\mathbb Z}
\newcommand{\cP}{\mathcal{P}}
\newcommand{\cR}{\mathcal{R}}
\newcommand{\cS}{\mathcal{S}}
\newcommand{\fsr}{finite subdivision rule }
\newcommand{\SP}{S_{\mathcal{P}}}
\newcommand{\SR}{S_{\mathcal{R}}}
\newcommand{\expm}{\varphi}
\newcommand{\subm}{\sigma_{\mathcal{R}}}

\newcommand{\pf}{\noindent {\bf Proof: }}

\newcommand{\nosubsections}{\renewcommand{\thethm}{\thesection.\arabic{thm}}
           \setcounter{thm}{0}}

\newcommand{\bdry}{\partial}

\title{Growth series for expansion complexes}

\author{J. W. Cannon}
\address{Department of Mathematics\\ Brigham Young University\\
Provo, UT 84602\\ U.S.A.}
\email{cannon@math.byu.edu}

\author{W. J. Floyd}
\address{Department of Mathematics\\ Virginia Tech\\
Blacksburg, VA 24061\\ U.S.A.}
\email{floyd@math.vt.edu}
\urladdr{http://www.math.vt.edu/people/floyd}

\author{W. R. Parry}
\address{Department of Mathematics\\ Eastern Michigan University\\
Ypsilanti, MI 48197\\ U.S.A.}
\email{walter.parry@emich.edu}

\date\today
\begin{abstract}
This paper is concerned with growth series for expansion complexes for
finite subdivision rules. Suppose $X$ is an expansion complex for a finite 
subdivision rule $\mathcal{R}$ with bounded valence and mesh approaching 
$0$, and let $S$ be a seed for $X$. One can define a growth series for 
$(X,S)$ by giving the tiles in the seed norm $0$ and then using either 
the skinny path norm or the fat path norm to recursively define norms for 
the other tiles. The main theorem is that, with respect to either of these 
norms, the growth series for $(X,S)$ has polynomial growth. Furthermore, 
the degrees of the growth rates of hyperbolic expansion complexes are dense 
in the ray $[2,\infty)$.
\end{abstract}
\maketitle

Suppose $T$ is the set of tiles in a tiling of a plane and $S$ is a 
nonempty, finite subset of $T$ (often a single tile).
We give $T$ a metric by defining the distance $d(s,t)$ between two tiles
$s$ and $t$ to be the minimum nonnegative integer $n$, such that there is
a finite sequence $s_0, s_1, \dots, s_n$ of tiles such that $s = s_0$, 
$s_n = t$, and for $i\in \{1,\dots,n\}$ $s_{i-1} \cap s_{i}$ contains an 
edge. For each nonnegative integer $n$, let $a_n$ be the number of tiles 
whose distance from an element of $S$ is $n$. The growth series for $(T,S)$ 
is the power series $\sum_{n=0}^{\infty} a_n z^n$.

We are interested in the growth series that arise from tilings. Unless one
imposes additional structure on the tiling, the only requirement for
the growth series is that each $a_n \ge 0$. One can see this by starting with
the collection of circles in the plane with center the origin and radius a
positive integer. These circle decompose the plane into the union of a disk
and a countable family of annuli.  If $\sum_{n=0}^{\infty} a_n z^n$
is a power series with each $a_n> 0$, then one can 
subdivide the central
disk radially into $a_0$ tiles and for each $n >0$ one can 
radially subdivide the 
annuli bounded by circles of radii $n$ and $n+1$ into $a_n$ tiles. 
This produces a pair $(T,S)$ with growth series $\sum_{n=0}^{\infty} a_n z^n$.

By contrast, consider a tiling $T$ of the Euclidean or hyperbolic plane coming
from the images, under a cocompact group $G$ of isometries of the plane, of
a Dirichlet region $D$ for the action of $D$. Let the seed $S$ be the single 
tile $D$. Then the growth series for $(T,S)$ is the growth series for the
group $G$ with respect to the geometric generating set 
$\Sigma = \{g\in G\co g(D) \cap D\ \textrm{is an edge of}\ D\}$.
Cannon shows in \cite{Ccomb} that if $G$ is a cocompact discrete group of 
isometries of $\bH^n$, then with respect to a finite generating set for 
$G$ the growth series is rational. In \cite{Benson}, Benson proves the 
analogous result for groups of Euclidean isometries. 
In \cite{CW}, Cannon and Wagreich consider 
(1) the case that $D$ is a hyperbolic triangle whose angles are submultiples 
of $\pi$ and $G$ is the associated group. They explicitly compute the
rational growth function $f$ and prove that $f(1) = 1/\chi(G)$ 
and that all of the poles of $f$
lie in the unit circle except for a pair of positive reciprocal poles. 
They also consider (2) the case that $D$ is 
a hyperbolic polygon with $4g$ sides and $G$ is the fundamental group of a
closed orientable surface with genus $g\ge 2$, and prove that the rational
growth function $f$ has the same properties. There is extensive literature on
the special properties of the rational grouwth functions of hyperbolic
surface groups. See for example the papers of 
Bartholdi-Ceccherini-Silberstein \cite{BC}, Floyd \cite{F}, 
Floyd-Plotnick \cite{FPchi,FPsym,FPpq}, and Parry \cite{P}.

Inspired by the growth functions for surface groups, and motivated by
a question from Maria Ramirez Solano about the growth rate for the pentagonal
subdivision rule, we decided to consider growth series for expansion
complexes. Like the tilings coming from surface groups, expansion
complexes are essentially determined by a finite amount of combinatoiral data.
But the growth series are very different.
In Theorem~\ref{thm:growth} we prove that the associated growth functions
all have polynomial growth, so they rarely have rational growth.

\section{Finite subdivision rules and expansion complexes}\label{sec:fsr}

While a finite subdivison rule is defined dynamically, in essence it is a 
finite combinatorial procedure for recursively subdividing appropriate 
$2$-complexes. A  \fsr $\cR$ consists of (1) a finite 2-complex $\SR$, (2) 
a subdivision $\cR(\SR)$ of $\SR$, and (3) a continuous cellular map
$\subm\co \cR(\SR) \to \SR$ whose restriction to every open cell is a
homeomorphism. We further require that $\SR$ is the union of its closed
$2$-cells and each closed $2$-cell is the image of a 
polygon (called its tile type) with at least three edges 
by a continuous cellular map whose restriction to each open cell
is a homeomorphism. An $\cR$-complex is a 2-complex which is the union of its
closed 2-cells together with a structure map $f\co X \to \SR$; we require that
$f$ is a continuous cellular map whose restriction to each open cell is a
homeomorphism. The subdivision $\cR(\SR)$ of $\SR$ pulls back under $f$ to a subdivision $\cR(X)$ of $X$; $\cR(X)$ is an $\cR$-complex with structure map
$\subm \circ f \co \cR(X) \to \SR$. Since $\cR(X)$ is an $\cR$-complex, one
can subdivide it; this is how one can recursively subdivide complexes with
a \fsr.  See \cite{fsr} for the basic theory of finite subdivision
rules.

As a simple example, consider the pentagonal subdivision rule $\cP$
which was first described in \cite{fsr}. The subdivision complex $\SP$ has
a single vertex, a single edge, and a single face (which is the image of a 
pentagon). The subdivision of the tile type is shown in Figure~\ref{fig:pent}.

\begin{figure}[ht]
\centering
\includegraphics{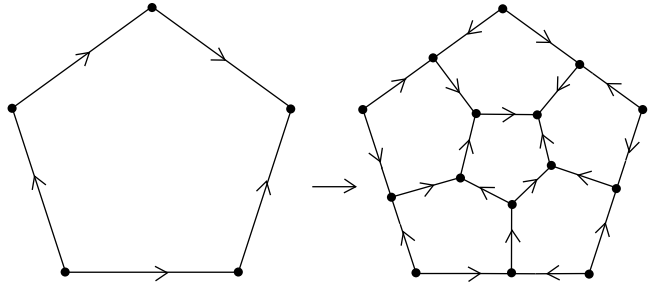}
\caption{The subdivision of the tile type for the 
pentagonal subdivision rule $\cP$}
\label{fig:pent}
\end{figure}

\begin{figure}[ht]
\centering
\includegraphics[width=1.5in]{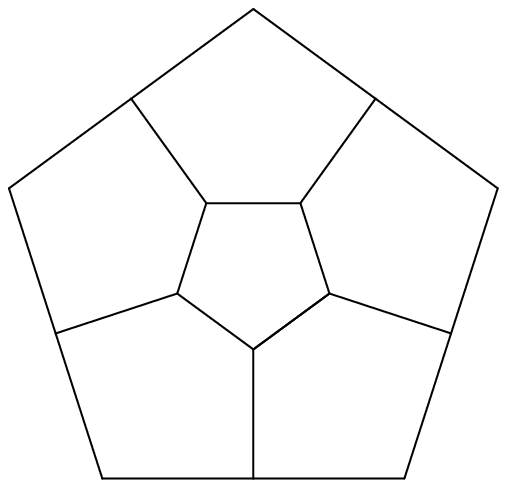}
\includegraphics[width=1.5in]{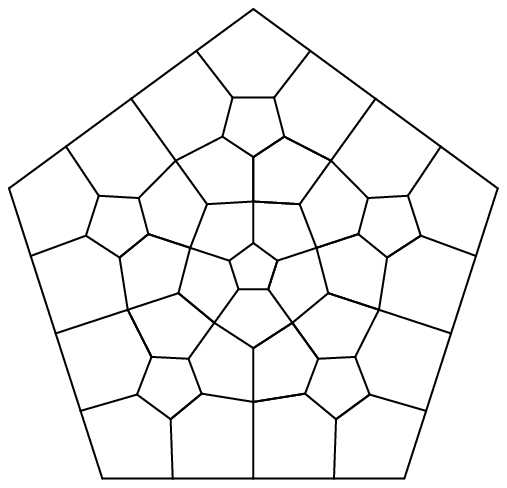}
\includegraphics[width=1.5in]{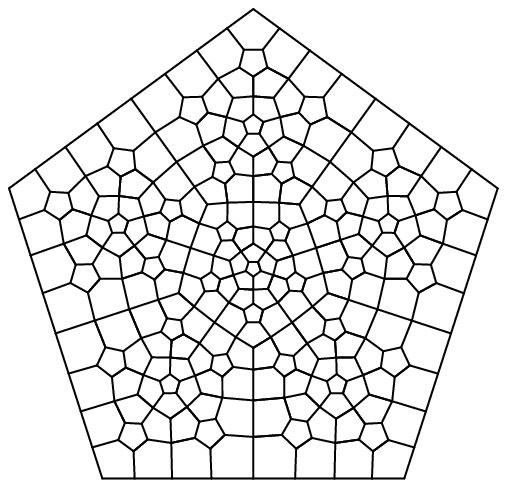}
\caption{$\cP(t)$, $\cP^2(t)$, and $\cP^3(t)$}
\label{fig:pent123}
\end{figure}

Bowers and Stephenson created the pentagonal expansion complex in \cite{BS} as
part of their analysis of the pentagonal subdivision rule.  
Figure~\ref{fig:pent123} shows the first three subdivisions of the tile 
type $t$, drawn with Stephenson's program
CirclePack \cite{CP}. We can identify $t$ with the central pentagon in its
first subdivision $\cP(t)$, and for each positive integer $n$ this induces
an inclusion of the $n$th subdivision $\cP^n(t)$ in
$\cP^{n+1}(t)$. The direct limit of the sequence of inclusions $\cP^n(t) \to
\cP^{n+1}(t)$ is the pentagonal expansion complex. Bowers and Stephenson
put a conformal structure on the pentagonal expansion complex by making each
open pentagon conformally a regular pentagon, using butterfiles to give charts
for open edges, and using power maps to give charts for vertices. They showed
that the expansion complex is conformally equivalent to the plane, and the 
expansion map (which takes each $\cP^n(t)$ to $\cP^{n+1}(t)$) is conformal.

In our papers \cite{expfsri, expfsrii} we gave the general definition of an
expansion complex for a finite subdivision rule and developed some of the
theory. An expansion complex for a \fsr $\cR$ is an $\cR$-complex $X$ which is
homeomorphic to $\bR^2$ such that there is an orientation-preserving 
homeomorphism $\expm \co X \to X$ such that $\subm \circ f = f \circ \expm$,
where $f$ is the structure map for $X$. If $X$ is an expansion complex and
$S$ is a subcomplex of $X$, then $S$ is a {\bf seed} of $X$ if $S$ is a
closed topological disk, $S \subset \expm(S)$, and $X = \cup_{n=0}^{\infty}
\expm^n (S)$. For the pentagonal expansion complex, one can take the tile type
$t$ to be a seed. It is possible for an expansion complex not to have a seed.
For example, if the subdivision map is the identity map then no expansion 
complex can have a seed. But if $\cR$ is a finite subdivision rule with bounded
valence and mesh approaching $0$, then it follows from 
\cite[Lemma 2.5]{expfsri} that every expansion complex for $\cR$ has a seed for
some iterate of $\cR$. As in \cite{BS} one can put a conformal structure on 
an expansion complex by taking each open tile to be conformally regular, 
using ``butterflies'' as charts in neighborhoods of open edges, and using
power maps to extend the conformal structure over the vertices. We call an
expansion complex \emph{parabolic} if with this conformal structure it is
conformally equivalent to the plane, and \emph{hyperbolic} if with this 
conformal structure it is conformally equivalent to the open unit disk. 

\section{Growth series for expansion complexes}

Let $\cR$ be a finite subdivision rule, let $X$ be an expansion complex for
$\cR$, and let $S$ be a seed for $X$. We define 
the \emph{skinny path norm} $|\cdot |$ on the tiles
of $X$ by setting norm $|t| = 0$ if $t$ is in $S$ and if $t$ is not in $S$ then
$|t|$ is the minimal positive integer $n$ such that there exist tiles
$t_0,\dots,t_n$ such that $t_0$ is in $S$, $t_n = t$, and 
$t_i \cap t_{i-1} \ne \emptyset$
for $1\le i \le n$.  For a nonnegative integer $n$, let
$s_n = \# \{\mbox{tiles } t\subset X\co |t| = n\}$ and let
$b_n = \# \{\mbox{tiles } t\subset X\co |t|\le n\}$ 
(so $s_n$ is the number of tiles
in the combinatorial sphere of radius $n$ and $b_n$ is the number of tiles in
the combinatorial ball of radius $n$).
The \emph{growth series} for $(X,S)$ is the 
power series $\sum_{n=0}^{\infty} b_n z^n$. 
The growth series has \emph{exponential growth} if
$\limsup_{n\to\infty} \sqrt[n] b_n > 1$ and has \emph{subexponential growth}
if $\limsup_{n\to\infty} \sqrt[n] b_n = 1$. The growth series has 
\emph{polynomial growth of degree $d$} if $d = \limsup_{n\to\infty}
\frac{\ln(b_n)}{\ln(n)}$. The growth series has \emph{intermediate
growth} if the growth is neither exponential nor polynomial.

\begin{thm}\label{thm:growth} 
Let $\cR$ be a finite subdivision rule with bounded valence and
mesh approaching $0$, let $X$ be a $\cR$-expansion complex, and let $S$ be a
seed for $X$. Then the growth series for $(X,S)$ with respect to the
skinny path norm has polynomial growth.
\end{thm}
\begin{proof}
We recall the skinny path distance from \cite{expfsri}. If $x,y\in X$, the 
\emph{skinny path distance} $d(x,y)$ is the minimum integer $n$ such that
there is a finite sequence $t_0,\dots,t_n$ of tiles
such that $x\in t_0$, $y\in t_n$,
and $t_{i-1} \cap t_n \ne \emptyset$ for $i\in \{1,\dots,n\}$. The skinny path
distance does not define a distance function on $X$ since two points in the 
same tile will have skinny path distance $0$, but it does define a
pseudometric.

Since $\cR$ has mesh approaching $0$, there is a positive 
integer $n_0$ such that the skinny path distance in $X$ from $S$ to $\partial
\expm^{n_0}(S)$ is at least $2$. By \cite[Lemma 2.7]{expfsri}, there is a postive
integer $n_1$ such that if $x,y \in X$ and $d(x,y) \ge 2$ then 
$d(\expm^{n_1}(x), \expm^{n_1}(y)) \ge 2d(x,y)$. 
It easily follows that there are a positive real
number $a$ and a real number $b > 1$ such that for every positive integer
$n$ the skinny path distance from $S$ to $\partial \expm^n(S)$  is greater than
$a b^n$. Since $S$ is compact and there is an upper bound, $d$, on the number
of subtiles in the first subdivision of a tile type of $\cR$, for any positive
integer $n$ the number of tiles in $\expm^n(S)$ is at most $c d^n$, where 
$c = \# S$.  Suppose $k\ge 2$ is
an integer such that $\ln(k) > \ln(a)$. Then there is a unique
positive integer $n$ such that $n -1 <  \frac{\ln(k)-\ln(a)}{\ln(b)} \le
n$. Then $ab^{n-1} < k \le ab^{n}$ and so
$$\frac{\ln(b_k)}{\ln(k)} < \frac{\ln(cd^n)}{\ln(ab^{n-1})}
= \frac{\ln(c) + n\ln(d)}{\ln(a/b) + n\ln(b)}$$ 
and so
$\limsup_{k\to\infty} \frac{\ln(b_k)}{\ln(k)} 
\le \frac{\ln(d)}{\ln(b)}$ and
the growth series has polynomial growth.
\end{proof}

One can also consider a growth series for $(X,S)$ with repect to the fat path
norm. As above, let 
$\cR$ be a finite subdivision rule, let $X$ be an expansion complex for
$\cR$, and let $S$ be a seed for $X$. We define
the \emph{fat path norm} $|\cdot |$ on the tiles
of $X$ by setting the norm $|t| = 0$ if $t$ is in $S$ and if 
$t$ is not in $S$ then
$|t|$ is the minimal positive integer $n$ such that there exist tiles
$t_0,\dots,t_n$ such that $t_0$ is in $S$, $t_n = t$, and
$t_i \cap t_{i-1}$ contains an edge for
$1\le i \le n$. The other definitions in the first paragraph of this
section follow exactly as before. Since for every nonnegative integer $n$
the number of tiles of fat path norm at most $n$ is at most the number of 
tiles of skinny path norm at most $n$, one gets the immediate corollary.

\begin{cor}\label{cor:growth}
Let $\cR$ be a finite subdivision rule with bounded valence and
mesh approaching $0$, let $X$ be a $\cR$-expansion complex, and let $S$ be a
seed for $X$. Then the growth series for $(X,S)$ with respect to the
fat path norm has polynomial growth.
\end{cor}

\section{A family of examples}\label{sec:examples}

In all of the examples we consider in this section, the skinny path norms and
the fat path norms are the same, so we won't name the norm.

We start with a simple example of an expansion complex for a finite 
subdivision rule $\cR_1$.  The subdivisions of the two tile types are shown
in Figure~\ref{fig:hyp-1-2sub}. The subdivision $\cR_1(t_1)$ of the tile type 
$t_1$ contains a tile in its interior which is labeled $t_1$, so the tile 
type $t_1$ is the seed of an expansion complex $X$.  Let $\expm\co X \to X$
be the expansion map. For convenience we denote the seed by $S$.
Figure~\ref{fig:hyp-1-2exp} shows part of the expansion complex, with the seed
in the center. Then $s_0 = 1$, $s_n = 2^{n+1}$ if $n >0$, and $b_n = 2^{n+2}-3$
for all $n\ge 0$. The growth series has exponential growth, but this doesn't
violate Theorem~\ref{thm:growth} since $\cR_1$ 
doesn't have mesh approaching $0$.

For each positive integer $n$, let $R_n = \expm^n(S) \setminus\textrm{int}
(S)$.  By \cite[Theorem 5.5]{expfsrii} $X$ is hyperbolic if
$\lim_{n\to\infty} M(R_n,\cS(X)) \ne \infty$, where $\cS(X)$ is the shingling
of $X$ by tiles.

Let $n$ be a positive integer. Define a weight function $w$ on $R_n$ as 
follows.  If $t$ is a tile of $R_n$, then for some $k\in\{1,\dots,n\}$,
$t\in \expm^k(S) \setminus \textrm{int}(\expm^{k-1}(S))$; 
we give $t$ weight $2^{n-k}$. 
The height curves for $R_n$ have height 
$H(R_n,w) = \sum_{i=0}^{n-1} 2^i = 2^n -1$, and
$w$ is the sum of the weights associated to the height curves. Hence 
by \cite[2.3.6]{SR} $w$ is the
optimal weight function for $R_n$ for fat flow modulus. For each $i\in
\{0,\dots,n-1\}$, there are $2^{n+1-i}$ tiles in $R_n$ with $w$-weight
$2^i$. Hence $A(R_n,w) = \sum_{i=0}^{n-1} 2^{n+1-i} \cdot (2^i)^2 =
2^{n+1}(2^n -1)$
and $$M(R_n,\cS(X)) = M(R_n,w) =  \frac{H(R_n,w)^2}{A(R_n,w)} = 
\frac{(2^n-1)}{2^{n+1}}.$$
Hence $\lim_{n\to\infty} M(R_n,\cS(X)) = \frac{1}{2} < \infty$ 
and $X$ is hyperbolic.

\begin{figure}[!ht]
\labellist
\small\hair 2pt
\pinlabel $t_1$ at 45 142
\pinlabel $t_1$ at 174 142
\pinlabel \rotatebox{-90}{$t_2$} at 144 142
\pinlabel \rotatebox{90}{$t_2$}  at 203 142
\pinlabel $t_2$ at 174 112
\pinlabel \rotatebox{180}{$t_2$} at 174 171
\pinlabel $t_2$ at 45 42
\pinlabel $t_2$ at 155 42
\pinlabel $t_2$ at 193 42
\pinlabel $\longrightarrow$ at 105 142
\pinlabel $\longrightarrow$ at 105 42
\endlabellist
\centering
\includegraphics{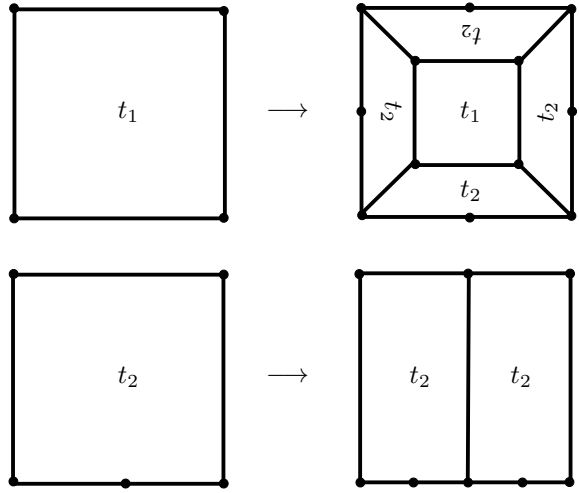} 
\caption{The subdivisions of the tile types for $\cR_1$}
\label{fig:hyp-1-2sub}
\end{figure}

\begin{figure}[!ht]
\centering
\includegraphics[width=4in]{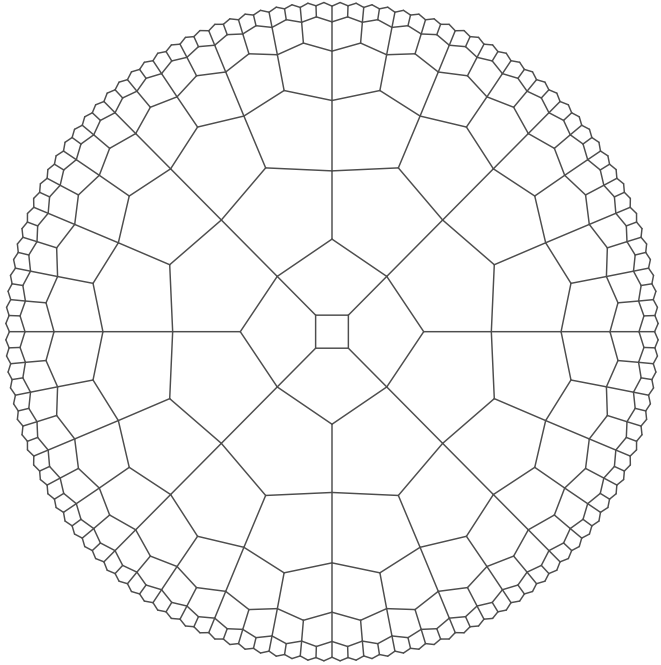}
\caption{Part of the expansion complex for $\cR_1$}
\label{fig:hyp-1-2exp}
\end{figure}

The finite subdivision rule $\cR_2$ is similar to $\cR_1$ but it has been
modified to have mesh approaching $0$. This time there are three tile types, 
and the subdivisions are shown in Figure~\ref{fig:hyp-2-3sub}. Tile type
$t_1$ is a seed for an expansion complex $X$; part of this expansion complex is
shown in Figure~\ref{fig:hyp-2-3exp}. One can show as we did for the previous
example that $X$ is hyperbolic. The hyperbolicity of $X$ also follows from
the proof of \cite[Lemma 5.1]{expfsrii}; this example is simpler than the 
example being analyzed there but the approach there fits this example as well.
Let $S$ be the seed of $X$ consisting of a single tile labeled $t_1$, 
and for each positive integer $n$ let
$R_n = \expm^n(S) \setminus\textrm{int}(S)$. Given $n$, define a weight
function $w$ on $R_n$ by giving a tile $t$ in $R_n$ weight $3^{n-k}$ if
$t\subset \expm^k(S) \setminus \textrm{int}\expm^{k-1}(S)$. The height
$H(R_n,w) = \sum_{k=0}^{n-1} 2^{n-1-k}3^k = 3^n - 2^n$. The weight function
$w$ is a sum of weight functions corresponding to height curves, so it
is an optimal weight function. The area $A(R_n,w) = 
4 \sum_{k=0}^{n-1} 3^{2k} 6^{n-1-k} = 4\cdot 3^{n-1} \cdot (3^n - 2^n)$,
so $M(R_n,\cS(X)) = M(R_n,w) = \frac{H(R_n,w)^2}{A(R_n,w)} = 
\frac{3^n - 2^n}{4\cdot 3^{n-1}}$,
$\lim_{n\to\infty} M(R_n,\cS(X)) = \frac{3}{4} < \infty$,
and $X$ is hyperbolic.

\begin{figure}[!ht]
\labellist
\small\hair 2pt
\pinlabel $t_1$ at 45 242
\pinlabel $t_1$ at 174 242
\pinlabel \rotatebox{-90}{$t_3$} at 144 242
\pinlabel \rotatebox{90}{$t_3$}  at 203 242
\pinlabel $t_3$ at 174 212
\pinlabel \rotatebox{180}{$t_3$} at 174 271
\pinlabel $t_2$ at 45 142
\pinlabel $t_2$ at 144 158
\pinlabel $t_2$ at 172 158
\pinlabel $t_2$ at 200 158
\pinlabel $t_2$ at 144 117
\pinlabel $t_2$ at 172 117
\pinlabel $t_2$ at 200 117
\pinlabel $t_3$ at 45 42
\pinlabel $t_2$ at 144 58
\pinlabel $t_2$ at 172 58
\pinlabel $t_2$ at 200 58
\pinlabel $t_3$ at 144 17
\pinlabel $t_3$ at 172 17
\pinlabel $t_3$ at 200 17
\pinlabel $\longrightarrow$ at 105 242
\pinlabel $\longrightarrow$ at 105 142
\pinlabel $\longrightarrow$ at 105 42
\endlabellist
\centering
\includegraphics{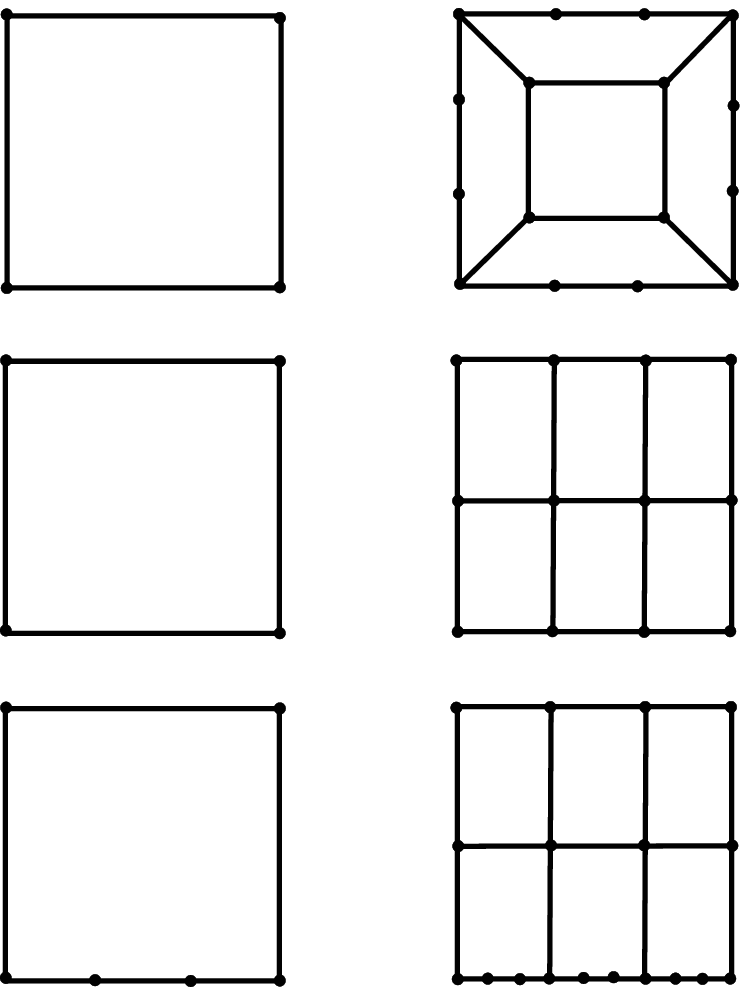}
\caption{The subdivisions of the tile types for $\cR_2$}
\label{fig:hyp-2-3sub}
\end{figure}

\begin{figure}[!ht]
\centering
\includegraphics[width=4in]{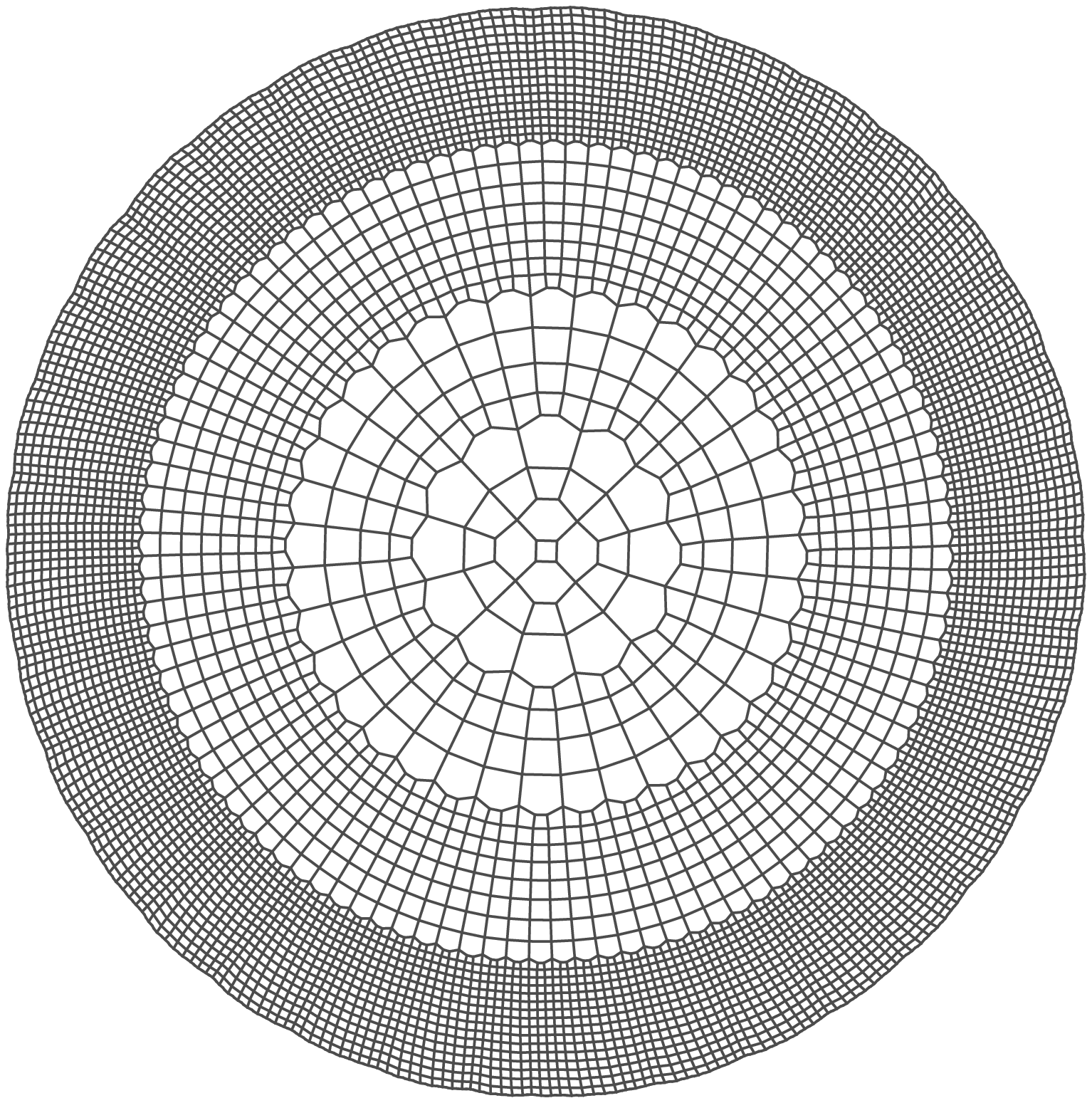}
\caption{Part of the expansion complex for $\cR_2$}
\label{fig:hyp-2-3exp}
\end{figure}

The finite subdivision rule $\cR_2$ is a special case ($\cR_{2,3}$)
of a two-parameter
family of finite subdivision rules $\cR_{p,q}$ for integers $p,q\ge 2$. For
a given $p$ and $q$, $\cR_{p,q}$ has three tile types, $t_1$ (a quadrilateral),
$t_2$ (a quadrilateral), and $t_3$ (a ($q$+3)-gon) which is viewed as a 
quadrilateral with the bottom edge subdivided into $q$ subedges. The tile type
$t_1$ is subdivided into $5$ subtiles, a central tile of type $t_1$ 
surrounded by four tiles of type $t_3$. The quadrilateral $t_2$ is 
subdivided into $pq$-subtiles, all of type $t_2$, arranged in $p$ rows 
and $q$ columns.  The tile type $t_3$ is also subdivided 
into $pq$ subtiles arranged in $p$ rows and $q$ columns, with
each column in the first $p-1$ rows containing a tile of type $t_2$ and each
column in the last row containing a tile of type $t_3$. 
As for the previous two examples, there is an expansion complex $X_{p,q}$ for 
$\cR_{p,q}$ whose seed $S$ is a single tile of type $t_1$. As before,
we denote the expansion map by $\expm$. For each
positive integer $n$ we let $R_n = \expm^n(S)\setminus\textrm{int}(S)$, and
we put a weight function on $R_n$ as follows: if $t$ is a tile of $R_n$ and
$t\subset \expm^k(S) \setminus \textrm{int}(\expm^{k-1}(S))$, then
the weight of $t$ is $q^{n-k}$. It follows as for $\cR_1$ and $\cR_2$ that $w$
is an optimal weight function for $M(R_n,\cS(X_{p,q}))$. If $p\ne q$, 
the height of $R_n$ with respect to $w$ is $H(R_n,w) = 
\sum_{k=0}^{n-1} q^k p^{n-1-k} = \frac{q^n-p^n}{q-p}$, the area is
$A(R_n,w) = 4\sum_{k=0}^{n-1} (pq)^{n-1-k}q^{2k} = 
\frac{4q^{n-1}(q^n-p^n)}{q-p}$, 
and the fat flow modulus of $R_n$ is 
$$\frac{H(R_n,w)^2}{A(R_n,w)} =
\frac{(q^n-p^n)^2}{(q-p)^2}\frac{q-p}{4q^{n-1}(q^n-p^n)}
= \frac{q^n-p^n}{4 q^{n-1}(q-p)} = \frac{1-(p/q)^n}{4(1-p/q)}.$$
If $p=q$, then $H(R_n,w) = n\cdot p^{n-1}$, $A(R_n,w) =  4n\cdot p^{2n-2}$,
and $M(R_n,w) = \frac{n}{4}$. If $p< q$, then $\limsup_{n\to\infty}
M(R_n,\cS(X_{p,q})) = \frac{q}{4(q-p)} < \infty$ and $X_{p,q}$ is 
hyperbolic.

We next look in more detail at the growth series for $X_{p,q}$. Suppose
$p,q\ge 2$, and consider $X_{p,q}$ as an expansion complex for $\cR_{p,q}$
with seed $S$ a single tile of type $t_1$. For looking at the finer detail of
the growth series, it is more convenient to look at the growth series
$g(z) = \sum_{n=0}^{\infty} s_n z^n$ for spheres instead of the growth series
$f(z) = \sum_{n=0}^{\infty} b_n z^n$ for balls. As we saw above,
$$g(z) = 1 + 4z + 4qz^2 + 4qz^3 + \dots + 4qz^{p+1} + 4q^2z^{p+2} 
+ \dots 4q^2 z^{p^2+p+1}
+ 4q^3 z^{p^2 + p + 2} + \dots,$$
where each coefficient $4q^k$ appears $p^k$ consecutive times.
For example, when $p=2$ and $q=3$ (the example $\cR_2$), 
$$g(z) = 1 + 4z + 12z^2 + 12 z^3 + 36 z^4 + \dots + 36z^7 +
108z^8 + \dots + 108 z^{15} + 324 z^{16} + \dots.$$
Since the sequence $\{s_n\}$ has no upper bound on the number of consecutive
terms which are constant, $g(z)$ cannot be rational or even $D$-finite.
However, $g(z)$ does satisfy a functional equation.
Note that
$$q\cdot g(z^p)= q  + 4qz^p + 4q^2(z^{2p} + z^{3p} + \dots + z^{p^2+p})  +
4q^3 z^{p^2 + 2p} + \dots,$$ so
$$q\cdot \left[ g(z^p)-1 \right] (1 + z + \dots + z^{p-1}) =z^{p-2}(g(z) - 1 - 4z)$$
and $g$ satisfies the functional equation
$$q\cdot (g(z^p)-1) \frac{(z^p-1)}{(z-1)} = z^{p-2}(g(z) - 1 - 4z).$$ 

We now consider the growth series $\sum_{n=0}^{\infty} b_n z^n$ for $X_{p,q}$
with respect to the seed $S$ consisting of a single tile labeled $t_1$.
For convenience we assume that $p,q\ge 2$.
Let $n$ be a nonnegative integer. Then there is a nonnegative integer $k$
such that $\frac{p^k-1}{p-1} \le n < \frac{p^{k+1}-1}{p-1}$. Let $m =
n - \frac{p^k-1}{p-1}$. Then $0 \le m \le p^k -1$.
Then $$n = \frac{p^k-1}{p-1} + m\quad \textrm{and} \quad
b_n = 1 + 4\frac{(pq)^k-1}{pq-1} + m(4q^k).$$
When $m = 0$, 
$$\frac{\ln(b_n)}{\ln(n)} = \frac{\ln(pq + 4(pq)^k -5) - \ln(pq-1)}{\ln(p^k-1) 
- \ln(p-1)}$$
and in general
$$\frac{\ln(b_n)}{\ln(n)} < \frac{\ln(pq + 4(pq)^{k+1} -5) - 
\ln(pq-1)}{\ln(p^{k+1}-1) - \ln(p-1)}.$$
It follows that the growth series has polynomial growth of
degree 
$$\limsup_{n\to\infty} \frac{\ln(b_n)}{\ln(n)} = \frac{\ln(pq)}{\ln(p)}
= 1 + \frac{\ln(q)}{\ln(p)}.$$
Since $X_{p,q}$ is hyperbolic whenever $q > p$, the degrees of the polynomial
growth rates of hyperbolic expansion complexes with respect to the fat path 
norm are dense in $[2,\infty)$, and the degrees of the polynomial
growth rates of hyperbolic expansion complexes with respect to the skinny path
norm are dense in $[2,\infty)$.

In his Ph.D. thesis \cite{Wood}, Wood notes that hyperbolic complexes can have 
spherical growth rates of degree $1+\epsilon$ for $\epsilon$ arbitrarily small.


\begin{thebibliography}{12}

\bibitem[Ben83]{Benson}
M.~Benson, \emph{Growth series of finite extensions of $\bZ^n$ are rational},
Invent. math. \textbf{73} (1983), 251--269.

\bibitem[BarC02]{BC}
L.~Bartholdi and T.~G.~Ceccherini-Silberstein,
\emph{Salem numbers and growth series of some hyerbolic graphs},
Geom, Dedicata \textbf{90} (2002), 107--114.
arXiv:math/9910067

\bibitem[BowS97]{BS}
P.~L. Bowers and K. Stephenson, \emph{A ``regular'' pentagonal tiling of
the plane}, Conform. Geom. Dyn. \textbf{1} (1997), 58--68 (electronic).

\bibitem[Can84]{Ccomb}
J.~W.~Cannon,
\emph{The combinatorial structure of cocompact discrete hyperbolic groups},
Geom. Dedicata \textbf{16} (1984), 123--148.

\bibitem[CanFP94]{SR}
J.~W. Cannon, W.~J. Floyd, and W.~R. Parry,
\emph{Squaring rectangles: the finite Riemann mapping theorem},
The mathematical legacy of Wilhelm Magnus:
groups, geometry and special functions (Brooklyn, NY, 1992), Amer. Math.
Soc., Providence, RI, 1994, pp.~133--212.

\bibitem[CanFP01]{fsr}
J.~W.~Cannon, W.~J.~Floyd, and W.~R.~Parry, \emph{Finite subdivision
rules}, Conform. Geom. Dyn. \textbf{5} (2001), 153--196
(electronic).

\bibitem[CanFP06a]{expfsri}
J.~W.~Cannon, W.~J.~Floyd, and W.~R.~Parry, \emph{Expansion
complexes for finite subdivision rules I}, Conform. Geom. Dyn.
\textbf{10} (2006), 63--99 (electronic).

\bibitem[CanFP06b]{expfsrii}
J.~W.~Cannon, W.~J.~Floyd, and W.~R.~Parry, \emph{Expansion
complexes for finite subdivision rules II}, Conform. Geom. Dyn.
\textbf{10} (2006), 326--354 (electronic).

\bibitem[CanW92]{CW}
J.~W.~Cannon and Ph.~Wagreich, \emph{Growth functions of surface
groups}, Math. Ann. \textbf{293} (1992), 239--257

\bibitem[Flo92]{F} 
William J.~Floyd,
\emph{Growth of planar Coxeter groups P.V.~numbers, and Salem numbers},
Math. Ann. \textbf{293} (1992), 475--483.

\bibitem[FloP87]{FPchi}
William~J.~Floyd and Steven~P.~Plotnick,
\emph{Growth functions on Fuchsian groups and the Euler characteristic},
Invent. math. \textbf{88} (1987), 1--29.

\bibitem[FloP88]{FPsym}
William~J.~Floyd and Steven~P.~Plotnick,
\emph{Symmetries of planar growth functions},
Invent. math. \textbf{93} (1988), 501--543.

\bibitem[FloP94]{FPpq}
William J.~Floyd and Steven P.~Plotnick,
\emph{Growth functions for semi-regular tilings of the hyperbolic plane},
Geom. Dedicata \textbf{53} (1994), 1--23.

\bibitem[Par93]{P}
Walter Parry,
\emph{Growth series of Coxeter groups and Salem numbers},
J.~Algebra \textbf{154} (1993), 406--415.

\bibitem[Ste]{CP}
K. Stephenson,
\emph{CirclePack}, software, available from\\
http://www.math.utk.edu/\~{}kens.

\bibitem[Woo06]{Wood} W.~E.~Wood, \emph{Combinatorial type problems for
triangulation graphs}, Ph.D. thesis, Florida State University, 2006.
\end{thebibliography}
\end{document}